\documentclass[10pt,leqno]{article}

\usepackage{amsmath,amssymb,amsthm,mathrsfs,dsfont}

\usepackage[margin=3cm]{geometry} 

\usepackage{titlesec,hyperref}

\usepackage{color}

\usepackage{fancyhdr}
\titleformat{\subsection}{\it}{\thesubsection.\enspace}{1pt}{}

\newtheorem{theo}{Theorem}[section]
\newtheorem{lemm}[theo]{Lemma}
\newtheorem{defi}[theo]{Definition}

\newtheorem{prop}[theo]{Proposition}

\numberwithin{equation}{section}

\allowdisplaybreaks 

\begin{document}
\title{Local well-posedness and global existence for the Popowicz system}

\author{
    Wei $\mbox{Tan}^{1,2}$ \footnote{Email: tanwei1008@126.com} \quad and\quad
    Zhaoyang $\mbox{Yin}^{1,3}$ \footnote{Corresponding author.E-mail: mcsyzy@mail.sysu.edu.cn}\\
    $^1\mbox{Department}$ of Mathematics, Sun Yat-sen University,\\
    Guangzhou, 510275, China\\
    $^2\mbox{College}$ of Mathematics and Statistics, Jishou University,\\
     Jishou, 416000, China\\
    $^3\mbox{Faculty}$ of Information Technology,\\
    Macau University of Science and Technology, Macau, China}

\date{}
\maketitle
\hrule

\begin{abstract}
Popowicz system, as the interacting system of Camassa-Holm and Degasperis-Procesi equations, has attracted some attention in recent years. In this paper,
   we first study the local well-posedness for the cauchy problem of Popowicz system in nonhomogeneous Besov spaces $B^s_{p,r}\times B^s_{p,r}$ with $s> \max\{2, \frac{1}{p}+\frac{3}{2}\}$
   or $(s=2, 2\leq p \leq \infty, 1\leq r\leq 2)$. Moreover, a new blow-up criterion and global existence with different initial values are obtained.

\vspace*{5pt}

\vspace*{5pt}
\end{abstract}

\vspace*{10pt}

\tableofcontents

\section{Introduction}
The aim of this paper is to study the local well-posedness in Besov spaces and global existence with different initial values for the Popowicz system:
\begin{equation}\label{eq1}
 \left\{
  \begin{array}{l}
  m_{t}+(2u+v)m_x+3(2u_x+v_x)m=0,\\
  n_t+(2u+v)n_x+2(2u_x+v_x)n=0,\\
  u(0,x)=u_0(x),v(0,x)=v_{0}(x),
  \end{array}
 \right.
\end{equation}
where $m=u-u_{xx}$ and $n=v-v_{xx}$. The system \eqref{eq1} was first derived recently by Z. Popowicz in \cite{1} and \cite{2},
based on the observation that the second Hamiltonian operator of the Degasperis-Procesi equation could be considered as the Dirac reduced Poisson
tensor of the second Hamiltonian operator of the Boussinesq equation.
It can be regarded as a system coupled by the camassa-Holm (C-H) and Degasperis-Procesi (D-P) equations,
so it is also called interacting system of the C-H and D-P equations by some scholars \cite{3, 4, 5}.

Actually, when $m=u=0$, Popowicz system \eqref{eq1} is simplified to the famous C-H equation, which is completely integrable.
As a typical shallow water wave equation \cite{6,7}, many different properties of C-H equation are studied by using different methods and techniques.
Constantin studied its bi-Hamiltonian structure \cite{8} and obtained its exact peaked solitons of the form $ce^{-|x-ct|}$ with $c>0$ \cite{9}.
The local well-posedness for the Cauchy problem of the C-H equation in Sobolev spaces and Besov spaces was discussed in \cite{10,11,12,13}.
It was shown that there exist global strong solutions to the C-H equation and finite time blow-up strong solutions to the C-H equation \cite{10,11,14,15}.
The existence and uniqueness of global weak solutions and global conservative and dissipative solutions were also inverstigated in \cite{16,17,18,19}.
When $n=v=0$, the system \eqref{eq1} becomes the famous D-P equation by rescaling the time variable.
As another typical shallow water wave model, D-P equation has also attracted extensive attention. Its traveling wave solution and soliton solution are reported respectively \cite{20,21}. The local well-posedness, blow-up phenomena, global strong solutions and global weak solutions were studied in
\cite{22,23,24,25,26,27,28,29,30,31}. Although the D-P equation is similar to the C-H equation in several aspects, these two equations are truly different.
In fact, the D-P equation has not only peakon solutions \cite{32} and periodic peakon solutions \cite{33} but also shock peakons \cite{34} and periodic shock waves \cite{24}.

However, to the best of our knowledge, there are few studies on the related properties of system \eqref{eq1}.
Barnes and Hone studied its conservative peakons and explicit dynamics of two peakons \cite{35}.
The local well-posedness of strong solutions and blow-up phenomena in Sobolev spaces are discussed in \cite{4, 5}.
Recently, Zhou studied the local well-posedness of solutions for this system \eqref{eq1} in nonhomogeneous Besov spaces $B^s_{p,r}$ with $1\leq p, r\leq +\infty$ and $s>\max\{\frac{5}{2}, 2+\frac{1}{p}\}$,
and established the local well-posedness in critical Besov space $B^{\frac{5}{2}}_{2,1}$ by the transport equations theory and the classical Friedrichs regularization method \cite{3}.
Besides, Zhou also reported the non-uniform dependence on initial data of system \eqref{eq1} by constructing two sequences of  bounded approximate solutions in
Sobolev space $H^s(\mathbb{R})$ with $s>\frac{5}{2}$ \cite{3}.

In this paper, we aim to improve the local well-posedness for the Cauchy problem of \eqref{eq1} in Besov spaces,
present a new blow-up criterion and the global existence with different initial values for the \eqref{eq1}.
The structure of the paper is as follows: in Section 2, we introduce some preliminaries which will be used in sequel.
In Section 3, we give an inequality which plays an important role in the proof of this paper,
and prove the local well-posedness of \eqref{eq1} in $B^s_{p,r}$ with $s>\max\{2, \frac 3 2+\frac 1 p\}$ or $(s=2, 2\leq p \leq\infty, 1\leq r\leq2)$.
The main method is based on the Littlewood-Paley theory and transport equations theory.
In Section 4, we obtain a new blow-up criterion and the global existence with different initial values.

\section{Preliminaries}
 \par
In this section, we will first  recall some important properties of the Littlewood-Paley decomposition and the nonhomogeneous Besov spaces $B^s_{p,r}$,
which will be used in this paper.

\begin{prop}\cite{book}
Let $\mathcal{C}\overset{def}{=}\{\xi\in\mathbb{R}^d:\frac 3 4\leq|\xi|\leq\frac 8 3\}$ be an annulus and $\mathcal{B}\overset{def}{=}\{\xi\in\mathbb{R}^d: |\xi|\leq\frac 4 3\}$ be a ball.
There exists a couple of smooth functions ($\chi,\varphi$) valued in [0, 1], such that $\chi \in C_c^{\infty}(\mathcal{B})$ and $\varphi \in C_c^{\infty}(\mathcal{C})$. Moreover,
$$ \forall\xi\in\mathbb{R}^d,\ \chi(\xi)+\sum_{j\geq 0}\varphi(2^{-j}\xi)=1, $$
$$ \forall\xi\in\mathbb{R}^d\backslash\{0\},\ \sum_{j\in\mathbb{Z}}\varphi(2^{-j}\xi)=1, $$
and
$$\mathrm{Supp}\ \varphi(2^{-j}\cdot)\cap \mathrm{Supp}\ \varphi(2^{-j'}\cdot)=\emptyset\,\ \text{if}\,\ |j-j'|\geq 2,$$
$$\mathrm{Supp}\ \chi(\cdot)\cap \mathrm{Supp}\ \varphi(2^{-j}\cdot)=\emptyset\,\ \text{if}\,\ j \geq 1.$$
\end{prop}

Denote $\mathcal{F}$ by the Fourier transform and $\mathcal{F}^{-1}$ by its inverse.
Let $u$ be a tempered distribution in $\mathcal{S}'(\mathbb{R}^d)$. For all $j\in\mathbb{Z}$,
we can define the nonhomogeneous dyadic blocks $\Delta_j$ and nonhomogeneous low frequency cut-off
operator $S_j$ as follows:

$$ \Delta_j u=0\,\ \text{if}\,\ j\leq -2,\quad \Delta_{-1} u=\mathcal{F}^{-1}(\chi\mathcal{F}u),$$
$$\Delta_j u=\mathcal{F}^{-1}(\varphi(2^{-j}\cdot)\mathcal{F}u)\,\ \text{if}\,\ j\geq 0,$$
$$S_j u=\sum_{j'=-\infty}^{j-1}\Delta_{j'}u.$$
Then the Littlewood-Paley decomposition is given as follows:
$$ u=\sum_{j\in\mathbb{Z}}\Delta_j u \quad \text{in}\ \mathcal{S}'(\mathbb{R}^d). $$

\begin{defi}\cite{book}
Let $s\in\mathbb{R}$ and $\ 1\leq p,r\leq\infty.$ The nonhomogeneous Besov space $B^s_{p,r}(\mathbb{R}^d)$ consists of all tempered distribution $u$ such that
$$ B^s_{p,r}=B^s_{p,r}(\mathbb{R}^d)=\{u\in S'(\mathbb{R}^d):\|u\|_{B^s_{p,r}(\mathbb{R}^d)}=\Big\|(2^{js}\|\Delta_j u\|_{L^p})_j \Big\|_{l^r(\mathbb{Z})}<\infty\}. $$
\end{defi}

Now, we give some properties about Besov spaces that will be used in this paper.
\begin{prop}\cite{book}
Suppose that $s\in\mathbb{R},\ 1\leq p, p_i, r, r_i\leq\infty, i=1,2.$ We have \\
(1) Topological properties: $B^s_{p,r}$ is a Banach space which is continuously embedded in $\mathcal{S}'(\mathbb{R}^d)$. \\
(2) Density: If $1\leq p,r<\infty$, then $C_c^{\infty}$ is dense in  $B^s_{p,r}$. If $r<\infty$, then $\lim\limits_{j\rightarrow\infty}\|S_j u-u\|_{B^s_{p,r}}=0$. \\
(3) Dual properties: If $1\leq p,r<\infty$, $B^s_{p',r'}=(B^s_{p,r})'.$\\
(4) Embedding: If $1\leq p_1\leq p_2 \leq \infty $ and $1\leq r_1\leq r_2 \leq \infty$, then $ B^s_{p_1,r_1}\hookrightarrow B^{s-d(\frac 1 {p_1}-\frac 1 {p_2})}_{p_2,r_2}. $
If $s_1<s_2$, then the embedding $B^{s_2}_{p,r_2}\hookrightarrow B^{s_1}_{p,r_1}$ is locally compact. \\
(5) Algebraic properties: $B^s_{p,r}$ is an algebra $\Leftrightarrow$ $B^s_{p,r}\hookrightarrow L^{\infty} \Leftrightarrow s>\frac d p\ \text{or}\ s=\frac d p,\ r=1.\quad $ \\
(6) Fatou property: if $(u_n)_{n\in\mathbb{N}}$ is a bounded sequence in $B^s_{p,r}$, then an element $u\in B^s_{p,r}$ and a subsequence $(u_{n_k})_{k\in\mathbb{N}}$ exist such that
$$ \lim_{k\rightarrow\infty}u_{n_k}=u\ \text{in}\ \mathcal{S}'\quad \text{and}\quad \|u\|_{B^s_{p,r}}\leq C\liminf_{k\rightarrow\infty}\|u_{n_k}\|_{B^s_{p,r}}. $$
(7) Let $m\in\mathbb{R}$ and $f$ be a $S^m$-mutiplier, (that is $f: \mathbb{R}^d \rightarrow \mathbb{R}$ is a smooth function and satisfies that $\forall\alpha\in\mathbb{N}^d$, $\exists C=C(\alpha)$, such that $|\partial^{\alpha}f(\xi)|\leq C(1+|\xi|)^{m-|\alpha|},\ \forall\xi\in\mathbb{R}^d)$.
Then the operator $f(D)=\mathcal{F}^{-1}(f\mathcal{F})$ is continuous from $B^s_{p,r}$ to $B^{s-m}_{p,r}$.
\end{prop}

\begin{prop}\label{prop}\cite{book, he}
(1) If $s_1<s_2$, $\theta \in (0,1)$, and $(p,r)$ is in $[1,\infty]^2$, then we have
$$ \|u\|_{B^{\theta s_1+(1-\theta)s_2}_{p,r}}\leq \|u\|_{B^{s_1}_{p,r}}^{\theta}\|u\|_{B^{s_2}_{p,r}}^{1-\theta}. $$
(2) If $s\in\mathbb{R},\ 1\leq p\leq\infty,\ \varepsilon>0$, a constant $C=C(\varepsilon)$ exists such that
$$ \|u\|_{B^s_{p,1}}\leq C\|u\|_{B^s_{p,\infty}}\ln\Big(e+\frac {\|u\|_{B^{s+\varepsilon}_{p,\infty}}}{\|u\|_{B^s_{p,\infty}}}\Big). $$
\end{prop}

\begin{lemm}\label{product}\cite{book,he}
(1) For any $s>0$ and any $(p,r)$ in $[1,\infty]^2$, the space $L^{\infty} \cap B^s_{p,r}$ is an algebra, and a constant $C=C(s,d)$ exists such that
$$ \|uv\|_{B^s_{p,r}}\leq C(\|u\|_{L^{\infty}}\|v\|_{B^s_{p,r}}+\|u\|_{B^s_{p,r}}\|v\|_{L^{\infty}}). $$
(2) If $1\leq p,r\leq \infty,\ s_1\leq s_2,\ s_2>\frac{d}{p} (s_2 \geq \frac{d}{p}\ \text{if}\ r=1)$ and $s_1+s_2>\max(0, \frac{2d}{p}-d)$, there exists $C=C(s_1,s_2,p,r,d)$ such that
$$ \|uv\|_{B^{s_1}_{p,r}}\leq C\|u\|_{B^{s_1}_{p,r}}\|v\|_{B^{s_2}_{p,r}}. $$
(3) If $1\leq p\leq 2$,  there exists $C=C(p,d)$ such that
$$ \|uv\|_{B^{\frac d p-d}_{p,\infty}}\leq C \|u\|_{B^{\frac d p-d}_{p,\infty}}\|v\|_{B^{\frac d p}_{p,1}}. $$
\end{lemm}

Secondly, some properties of nonhomogeneous Bony decomposition that need to be used in this paper are introduced.

Let $u$ and $v$ be tempered distribution in $\mathcal{S}'(\mathbb{R}^d)$. Then we have
$$u=\sum_{j'}\Delta_{j'}u, \quad v=\sum_{j}\Delta_{j}v \quad and \quad uv=\sum_{j',j}\Delta_{j'}u\Delta_{j}v. $$ The nonhomogeneous paraproduct of $v$ by $u$ is defined as follows:
$$T_{u}v=\sum_{j}S_{j-1}u\Delta_{j}v.$$
The nonhomogeneous remainder of $u$ and $v$ is defined by
$$R(u,v)=\sum_{|j'-j|\leq 1}\Delta_{j'}u\Delta_{j}v.$$
At least formally, the operators $T$ and $R$ are bilinear, and we have the following Bony decomposition:
$$uv=T_{u}v+T_{v}u+R(u,v).$$

We now state our main result concerning continuity of the nonhomogeneous paraproduct operator $T$ and remainder operator $R$.
\begin{prop}\cite{book}
There exists a constant C such that for any real number s and any $1\leq p, r \leq +\infty$, we have, for any $(u, v)$ in $L^{\infty}\times B_{p,r}^{s}$,
$$\|T_{u}v\|_{B_{p,r}^{s}}\leq C^{1+|s|}\|u\|_{L^{\infty}}\|v\|_{B_{p,r}^{s}}.$$
Moreover, for any $(s,t)$ in $ \mathbb{R} \times (-\infty, 0)$ and any $1\leq p, r_1, r_2 \leq \infty$, we have, for any $(u,v)\in B_{\infty,r_1}^{t}\times B_{p,r_2}^{s}$,
$$\|T_{u}v\|_{B_{p,r}^{s+t}}\leq \frac{C^{1+|s+t|}}{-t}\|u\|_{B_{\infty, r_{1}}^{t}}\|v\|_{B^{s}_{p, r_{2}}}\,\ \text{with}\,\
\frac{1}{r}=min\{1,\frac{1}{r_{1}}+\frac{1}{r_{2}}\}.$$
\end{prop}

\begin{prop}\cite{book}
A constant C exists which satisfies the following inequalities.
Let $(s_1,s_2)$ be in $\mathbb{R}$ and $1\leq p, p_1, p_2, r, r_1, r_2\leq\infty$. Assume that
$$\frac{1}{p}=\frac{1}{p_{1}}+\frac{1}{p_{2}}\leq 1 \,\ \text{and}\,\ \frac{1}{r}=\frac{1}{r_{1}}+\frac{1}{r_{2}}\leq 1.$$
If $s_1+s_2>0$, then we have, for any (u,v) in $B_{p_1,r_1}^{s_1}\times B_{p_2,r_2}^{s_2},$
 $$\|R(u,v)\|_{B_{p,r}^{s_1+s_2}}\leq \frac{C^{1+|s_1+s_2|}}{s_1+s_2}\|u\|_{B_{p_1, r_{1}}^{s_1}}\|v\|_{B^{s_2}_{p_2, r_{2}}}.$$
If $r=1$ and $s_1+s_2=0,$ we have, for any (u, v) in $B_{p_1,r_1}^{s_1}\times B_{p_2,r_2}^{s_2},$,
$$\|R(u,v)\|_{B_{p,\infty}^{0}}\leq C^{1+|s_1+s_2|}\|u\|_{B_{p_1, r_{1}}^{s_1}}\|v\|_{B^{s_2}_{p_2, r_{2}}}.$$
\end{prop}

Finally, we state some useful results in the transport equation theory, which are crucial to the proofs of our main theorem later.
\begin{equation}\label{transport}
\left\{\begin{array}{l}
    f_t+v\cdot\nabla f=g,\ x\in\mathbb{R}^d,\ t>0, \\
    f(0,x)=f_0(x).
\end{array}\right.
\end{equation}

\begin{lemm}\label{priori estimate}\cite{book,li2}
Let $s\in\mathbb{R},\ 1\leq p,r\leq\infty$. For all solutions $f\in L^{\infty}([0,T];B^s_{p,r})$ of \eqref{transport} in one dimension with initial data $f_0$ in $B^s_{p,r}$, and $g$ in $L^1([0,T];B^s_{p,r})$,
there exists a constant $C$ such that if $t\in[0,T]$, then we have
$$ \|f(t)\|_{B^s_{p,r}}\leq e^{CV(t)}\Big(\|f_0\|_{B^s_{p,r}}+\int_0^t e^{-CV(t')}\|g(t')\|_{B^s_{p,r}}dt'\Big) $$
with $V(t)=\int_0^tV'(s)ds$ and
\begin{equation*}
  V'(t)=\left\{\begin{array}{ll}
  \|\nabla v\|_{B^{s+1}_{p,r}},\ &\text{if}\ s>\max(-\frac 1 2,\frac 1 {p}-1), \\
  \|\nabla v\|_{B^{s}_{p,r}},\ &\text{if}\ s>\frac 1 {p}\ \text{or}\ (s=\frac 1 {p},\ p<\infty, \ r=1),\\
  \|\nabla v\|_{B^{\frac 1 p}_{p,1}}, \ &\text{if}\ s=\frac 1 p-1,\ 1\leq p\leq 2,\ r=\infty.
  \end{array}\right.
\end{equation*}
\end{lemm}

\begin{lemm}\label{new estimate}\cite{book}
Let $s>0,\ 1\leq r\leq\infty,\ 1\leq p\leq p_1\leq\infty$ and $\frac{1}{p_2}=\frac{1}{p}-\frac{1}{p_1}$. Define $R_j=[v\cdot\nabla, \Delta_j]f$. There exists a constant $C$ such that
$$\Big\|(2^{js}\|R_j\|_{L^p})_j\Big\|_{l^r(\mathbb{Z})}\leq C(\|\nabla v\|_{L^{\infty}}\|f\|_{B^s_{p,r}}+\|\nabla v\|_{B^{s-1}_{p_1,r}}\|\nabla f\|_{L^{p_2}}).$$
Hence, if $f$ solves the equation \eqref{transport}, we have
$$\|f(t)\|_{B^s_{p,r}}\leq \|f_0\|_{B^s_{p,r}}+C\int_0 ^t (\|\nabla v\|_{L^{\infty}}\|f\|_{B^s_{p,r}}+\|\nabla v\|_{B^{s-1}_{p_1,r}}\|\nabla f\|_{L^{p_2}}+\|g\|_{B^s_{p,r}}dt').$$
\end{lemm}

\begin{lemm}\label{existence}\cite{book}
Let $1\leq p\leq p_1\leq\infty,\ 1\leq r\leq\infty,\ s> -d\min(\frac 1 {p_1}, \frac 1 {p'})$. Let $f_0\in B^s_{p,r}$, $g\in L^1([0,T];B^s_{p,r})$, and let $v$ be a time-dependent vector field such that $v\in L^\rho([0,T];B^{-M}_{\infty,\infty})$ for some $\rho>1$ and $M>0$, and
$$
  \begin{array}{ll}
    \nabla v\in L^1([0,T];B^{\frac d {p_1}}_{p_1,\infty}), &\ \text{if}\ s<1+\frac d {p_1}, \\
    \nabla v\in L^1([0,T];B^{s-1}_{p,r}), &\ \text{if}\ s>1+\frac d {p_1}\ or\ (s=1+\frac d {p_1}\ and\ r=1).
  \end{array}
$$
Then the equation \eqref{transport} has a unique solution $f$ in \\
-the space $C([0,T];B^s_{p,r})$, if $r<\infty$, \\
-the space $\Big(\bigcap_{s'<s}C([0,T];B^{s'}_{p,\infty})\Big)\bigcap C_w([0,T];B^s_{p,\infty})$, if $r=\infty$.
\end{lemm}

\begin{lemm}\label{continuity}\cite{li}
Let $1\leq p\leq\infty,\ 1\leq r<\infty,\ s>1+\frac d p\ (or \ s=1+\frac d p,\ 1\leq p<\infty,\ r=1)$. Denote $\bar{\mathbb{N}}=\mathbb{N}\cup\{\infty\}$.
Let $(v^n)_{n\in\bar{\mathbb{N}}}$ be a sequence of functions belonging to $C([0,T];B^{s-1}_{p,r})$. Assume that $(f^n)_{n\in\bar{\mathbb{N}}}$ in $C([0,T];B^{s-1}_{p,r})$ is the solution to
\begin{equation}
\left\{\begin{array}{l}
    f^n_t+v^n\cdot\nabla f^n=g,\ x\in\mathbb{R}^d,\ t>0, \\
    f^n(0,x)=f_0(x)
\end{array}\right.
\end{equation}
with initial data $f_0\in B^{s-1}_{p,r},\ g\in L^1([0,T];B^{s-1}_{p,r})$ and that for some $\alpha\in L^1([0,T])$, $\sup\limits_{n\in\bar{\mathbb{N}}}\|v^n(t)\|_{B^{s}_{p,r}}\leq \alpha(t)$.
If $v^n \rightarrow v^{\infty}$ in $L^1([0,T];B^{s-1}_{p,r})$, then $f^n \rightarrow f^{\infty}$ in $C([0,T];B^{s-1}_{p,r})$.
\end{lemm}

\section{Local well-posedness}
\par
In this section, we will establish local well-posedness of \eqref{eq1} in Besov spaces.
We firstly provide the framework in which we shall reformulate system \eqref{eq1}. Note that if $p(x):=\frac{1}{2}e^{|x|}, x\in\mathbb{R}$,
then $(1-\partial_x^2)^{-1}f=p*f$ for all $f\in L^2(\mathbb{R})$ and $(p*m, p*n)=(u, v)$. Then, system \eqref{eq1} can be rewritten as follows:
\begin{equation}\label{eq31}
  \left\{
   \begin{array}{l}
   u_{t}+(2u+v)u_x=p*\big(-3(2u_x+v_x)u+v_xu_{xx}-v_{xx}u_{x}\big)\overset{\Delta}{=}F,\\
   v_{t}+(2u+v)v_x=p*\big(-2(2u_x+v_x)v-2v_xu_{xx}-v_{xx}v_{x}\big)\overset{\Delta}{=}H, \\
   u(0,x)=u_0(x),v(0,x)=v_{0}(x).
   \end{array}
  \right.
\end{equation}

In order to prove the local well-posedness of the strong solution of system \ref{eq31}, we give the definition of space $E^s_{p,r}(T)$ and a lemma which plays a key role in the proof.
\begin{defi}
Let $T>0,\ s\in\mathbb{R},$ and $1\leq p,r \leq\infty.$ Set
\begin{equation*}
    E^s_{p,r}(T)\triangleq \left\{\begin{array}{ll}
    C([0,T];B^s_{p,r})\cap C^1([0,T];B^{s-1}_{p,r}), & \text{if}\ r<\infty,  \\
    C_w([0,T];B^s_{p,\infty})\cap C^{0,1}([0,T];B^{s-1}_{p,\infty}), & \text{if}\ r=\infty.
    \end{array}\right.
\end{equation*}
\end{defi}

\begin{lemm}\label{lemma}
Let $s>max\{2,\frac{1}{p}+\frac{3}{2}\}$ (or $s=2,2\leq p \leq \infty, 1\leq r\leq 2$), for any $u, v \in \mathcal{S}'(\mathbb{R}^d)$, there exists a constant $C$ such that
$$
\|uv\|_{B^{s-3}_{p,r}}\leq C\|u\|_{B^{s-2}_{p,r}} \|v\|_{B^{s-2}_{p,r}}.
$$
\end{lemm}
Proof: The proof of this Lemma is divided into two parts according to the value of $s$. By virtue of bony's decomposition, we have
$$uv=T_{u}v+T_{v}u+R(u,v).$$
Firstly, we prove the case of $s>max\{2,\frac{1}{p}+\frac{3}{2}\}$.
According to Proposition 2.3 and 2.6, since $B_{p, r}^s \hookrightarrow B_{p, r_1}^s (r\leq r_1)$, we have
\begin{equation}\label{eq32}
\|T_{u}v\|_{B^{s-3}_{p,r}}\leq C\|u\|_{B^{-\frac{1}{p}}_{\infty,\infty}}\|v\|_{B^{s-3+\frac{1}{p}}_{p,r}}\leq C\|u\|_{B^{0}_{p,\infty}}\|v\|_{B^{s-2}_{p,r}}\leq C\|u\|_{B^{s-2}_{p,r}}\|v\|_{B^{s-2}_{p,r}},
\end{equation}
and
\begin{equation}\label{eq33}
\|T_{v}u\|_{B^{s-3}_{p,r}}\leq C\|v\|_{B^{-1}_{\infty,\infty}}\|u\|_{B^{s-2}_{p,r}}\leq C\|v\|_{B^{\frac{1}{p}-1}_{p,\infty}}\|v\|_{B^{s-2}_{p,r}}\leq C\|u\|_{B^{s-2}_{p,r}}\|v\|_{B^{s-2}_{p,r}}.
\end{equation}
When $2\leq p\leq\infty$, according to Propositions 2.3 and 2.7, we get
\begin{equation}\label{eq34}
\|R(u,v)\|_{B^{s-3}_{p,r}}\leq C\|R(u,v)\|_{B^{s-3+\frac{1}{p}}_{\frac{p}{2},r}}\leq C\|R(u,v)\|_{B^{2s-4}_{\frac{p}{2},r}}\leq C\|u\|_{B^{s-2}_{p,r}}\|v\|_{B^{s-2}_{p,r}}.
\end{equation}
When $1\leq p\leq 2$, combining Propositions 2.3 and 2.7, we have
\begin{equation}\label{eq35}
\begin{split}
\|R(u,v)\|_{B^{s-3}_{p,r}}&\leq C\|R(u,v)\|_{B^{s-2-\frac{1}{p}}_{1,r}}\leq C\|R(u,v)\|_{B^{s-\frac{3}{2}-\frac{1}{p}}_{1,r}}\leq C\|u\|_{B^{s-2}_{p,r}}\|v\|_{B^{\frac{1}{2}-\frac{1}{p}}_{p',\infty}}
\\
&\leq C\|u\|_{B^{s-2}_{p,r}}\|v\|_{B^{\frac{1}{2}-\frac{1}{p'}}_{p,\infty}}\leq C\|u\|_{B^{s-2}_{p,r}}\|v\|_{B^{s-2}_{p,r}}.
\end{split}
\end{equation}
Combining inequalities \eqref{eq32}, \eqref{eq33}, \eqref{eq34}and \eqref{eq35} yields
\begin{equation}\label{eq36}
\|uv\|_{B^{s-3}_{p,r}}\leq C\|u\|_{B^{s-2}_{p,r}} \|v\|_{B^{s-2}_{p,r}}.
\end{equation}

Secondly, when $s=2,2\leq p \leq \infty, 1\leq r\leq 2$, combining Propositions 2.3 and 2.6, since $B_{\infty, \infty}^{-1} \hookrightarrow  L^{\infty}$, we can obtain
\begin{equation}\label{eq37}
  \begin{array}{l}
\|T_{u}v\|_{B^{s-3}_{p,r}}\leq C\|u\|_{B^{-1}_{\infty,\infty}}\|v\|_{B^{s-2}_{p,r}}\leq C\|u\|_{B^{\frac{1}{p}-1}_{p,r}}\|v\|_{B^{s-2}_{p,r}}\leq C\|u\|_{B^{s-2}_{p,r}}\|v\|_{B^{s-2}_{p,r}},
\\
\|T_{v}u\|_{B^{s-3}_{p,r}}\leq C\|v\|_{L^{\infty}}\|u\|_{B^{s-3}_{p,r}}\leq C\|v\|_{B^{-1}_{\infty,\infty}}\|u\|_{B^{s-2}_{p,r}}\leq C\|v\|_{B^{\frac{1}{p}-1}_{p,r}}\|u\|_{B^{s-2}_{p,r}}\leq \|v\|_{B^{s-2}_{p,r}}\|u\|_{B^{s-2}_{p,r}}.
  \end{array}
\end{equation}
According to Propositions 2.3 and 2.7, since $l^r \hookrightarrow l^{r'}$ (where $r\leq r', 1=\frac{1}{r}+\frac{1}{r'}$), we have
\begin{equation}\label{eq38}
  \begin{array}{l}
  \|R(u,v)\|_{B^{s-3}_{p,r}}\leq C\|R(u,v)\|_{B^{s-3+\frac{1}{p}}_{\frac{p}{2},\infty}}\leq C\|R(u,v)\|_{B^{2s-4}_{\frac{p}{2},\infty}}\leq C\|u\|_{B^{s-2}_{p,r}}\|v\|_{B^{s-2}_{p,r'}}
  \leq C\|u\|_{B^{s-2}_{p,r}}\|v\|_{B^{s-2}_{p,r}}.
  \end{array}
\end{equation}
Combining inequalities \eqref{eq36}, \eqref{eq37} and \eqref{eq38} completes the proof.

\begin{theo}\label{theorem}
Let $1\leq p, r\leq\infty, s\in \mathbb{R}$, and $(s, p, r)$ satisfy the condition $s> \max\{2, \frac{1}{p}+\frac{3}{2}\}$ or $(s=2, 2\leq p \leq \infty, 1\leq r\leq 2)$. Assumed that the initial value $(u_{0}, v_{0})\in B^{s}_{p,r}\times B^{s}_{p,r}$, then there exists a time $T>0$ such that \eqref{eq31} has a unique solution $(u, v)\in E^{s}_{p,r}(T)\times E^{s}_{p,r}(T)$.
\end{theo}

\textbf{Proof:} In order to prove Theorem \ref{theorem}, we will prove it from the two parts of the existence and uniqueness of the solution. Firstly, we will prove the existence of the solution, which is mainly divided into the following four steps.

\textbf{Step 1:} Constructing approximate solutions

We firstly set $u^0\overset{\Delta}{=}0, v^0\overset{\Delta}{=}0$ and define a sequence $(u^n, v^n)_{n\in \mathbb{N}}$ of smooth functions by solving the following linear systems:
\begin{equation}\label{eq39}
 \left\{
  \begin{array}{l}
  u_{t}^{n+1}+(2u^{n}+v^{n})u_x^{n+1}=p*\big(-3(2u_x^{n}+v_x^{n})u^{n}+v_x^{n}u_{xx}^{n}-v_{xx}^{n}u_{x}^{n}\big)\overset{\Delta}{=}F_{n},\\
  v_{t}^{n+1}+(2u^{n}+v^{n})v_x^{n+1}=p*\big(-2(2u_x^{n}+v_x^{n})v^{n}-2v_x^{n}u_{xx}^{n}-v_{xx}^{n}v_{x}^{n}\big)\overset{\Delta}{=}H_{n},\\
  u^{n+1}|_{t=0}=S_{n+1}u_0,\quad v^{n+1}|_{t=0}=S_{n+1}v_0.
  \end{array}
 \right.
\end{equation}
Let $G^{n}=2u^{n}+v^{n}$, and assume that $ (u^{n}, v^n) \in L^{\infty}([0, T]; B_{p, r}^s \times B_{p, r}^s)$ for all positive $T$.
Since $p$ is an $S^{-2}-$multiplier. Noticing that $B_{p, r}^s$ is an algebra when $(s, p, r)$ satisfies the condition theorem \eqref{theorem} and combining Lemma 2.5,  we obtain
$$\|G_{x}^{n}\|_{B_{p,r}^{s-1}}\leq2\|u_{x}^{n}\|_{B_{p,r}^{s-1}}+\|v_{x}^{n}\|_{B_{p,r}^{s-1}}\leq C(\|u^{n}\|_{B_{p,r}^{s}}+\|v^{n}\|_{B_{p,r}^{s}}),$$
\begin{align*}
\|p*(-3(2u_x^{n}+v_x^{n})u^{n})\|_{B_{p,r}^{s}}&\leq C\|-3(2u_x^{n}+v_x^{n})u^{n})\|_{B_{p,r}^{s-2}}
\leq C\|-6u_x^{n}-3v_x^{n}\|_{B_{p,r}^{s-2}}\|u^{n}\|_{B_{p,r}^{s-1}}\\
&
\leq C(\|u^{n}\|_{B_{p,r}^{s}}+\|v^{n}\|_{B_{p,r}^{s}})\|u^{n}\|_{B_{p,r}^{s}},
\\
\|p*(v_x^{n}u_{xx}^{n})\|_{B_{p,r}^{s}}&\leq C\|v_x^{n}u_{xx}^{n}\|_{B_{p,r}^{s-2}}\leq C\|v_x^{n}\|_{B_{p,r}^{s-1}}\|u_{xx}^{n}\|_{B_{p,r}^{s-2}}\leq C\|v^{n}\|_{B_{p,r}^{s}}\|u^{n}\|_{B_{p,r}^{s}},
\\
\|p*(v_{xx}^{n}u_{x}^{n})\|_{B_{p,r}^{s}}&\leq C\|v_{xx}^{n}u_{x}^{n}\|_{B_{p,r}^{s-2}}\leq C\|v^{n}\|_{B_{p,r}^{s}}\|u^{n}\|_{B_{p,r}^{s}},
\end{align*}
thus, we get
\begin{equation}\label{eq40}
\|F_{n}\|_{B_{p,r}^{s}}\leq C(\|u^{n}\|_{B_{p,r}^{s}}+\|v^{n}\|_{B_{p,r}^{s}})\|u^{n}\|_{B_{p,r}^{s}}.
\end{equation}
Similarly, we can get
\begin{equation}\label{eq41}
\|H_{n}\|_{B_{p,r}^{s}}\leq C(\|u^{n}\|_{B_{p,r}^{s}}+\|v^{n}\|_{B_{p,r}^{s}})\|v^{n}\|_{B_{p,r}^{s}}.
\end{equation}
Then $G^n_x,\ F_n,\ H_n \in L^{\infty}([0,T];B^s_{p,r})$.
Hence Lemma \ref{existence} ensures that \eqref{eq39} has a global solution $(u^{n+1}, v^{n+1})$ which belongs to $E^s_{p,r}(T)\times E^s_{p,r}(T)$ for all $T>0$.

\textbf{Step 2:} Uniform bounds

Using Lemma 2.8 together with \eqref{eq40} and \eqref{eq41}, we have
$$
\|u^{n+1}(t)\|_{B_{p,r}^{s}}\leq e^{C\int_0^t \|G^n\|_{B^{s}_{p,r}}dt'}
    \Big(\|S_{n+1}u_0\|_{B^{s}_{p,r}}+\int_0^t e^{-C\int_0^{t'} \|G^n\|_{B^{s}_{p,r}} dt''}\|F^n\|_{B^{s}_{p,r}}dt'\Big).
$$
Let $A_n=\|u^n\|_{B^{s}_{p,r}}+\|v^n\|_{B^{s}_{p,r}}$. Then $\|G^n\|_{B^{s}_{p,r}}\leq CA_n$, we deduce
\begin{align}\label{eq312}
\|u^{n+1}(t)\|_{B_{p,r}^{s}}\leq e^{C\int_0^t A_ndt'}
    \Big(\|S_{n+1}u_0\|_{B^{s}_{p,r}}+\int_0^t e^{-C\int_0^{t'} A_n dt''}A_{n}\|u^n\|_{B^{s}_{p,r}}dt'\Big).
\end{align}
Similarly, we can get
\begin{align}\label{eq313}
\|v^{n+1}(t)\|_{B_{p,r}^{s}}\leq e^{C\int_0^t A_ndt'}
    \Big(\|S_{n+1}v_0\|_{B^{s}_{p,r}}+\int_0^t e^{-C\int_0^{t'} A_n dt''}A_{n}\|v^n\|_{B^{s}_{p,r}}dt'\Big).
\end{align}
The combination \eqref{eq312} and \eqref{eq313} yields
\begin{align}\label{eq314}
A_{n+1}(t)\leq e^{C\int_0^t A_ndt'}
    \Big(\|S_{n+1}u_0\|_{B^{s}_{p,r}}+\|S_{n+1}v_0\|_{B^{s}_{p,r}}+\int_0^t e^{-C\int_0^{t'} A_n dt''}A_{n}^2dt'\Big).
\end{align}
We fix a $T >0$ such that $2C(\|u_0\|_{B^{s}_{p,r}}+\|v_0\|_{B^{s}_{p,r}})T\leq 1,$ and suppose by induction that for all $t\in [0,T]$,
\begin{align}\label{eq315}
\|u^n\|_{B^{s}_{p,r}}+\|v^n\|_{B^{s}_{p,r}}\leq\frac{\|u_0\|_{B^{s}_{p,r}}+\|v_0\|_{B^{s}_{p,r}}}{1-2C(\|u_0\|_{B^{s}_{p,r}}+\|v_0\|_{B^{s}_{p,r}})t}.
\end{align}
Substituting \eqref{eq315} into \eqref{eq314} yields
\begin{align*}
\|u^{n+1}\|_{B^{s}_{p,r}}+\|v^{n+1}\|_{B^{s}_{p,r}}
&\leq \big(1-2C(\|u_0\|_{B^{s}_{p,r}}+\|v_0\|_{B^{s}_{p,r}})t\big)^{-\frac{1}{2}}\big(\|S_{n+1}u_0\|_{B^{s}_{p,r}}+\|S_{n+1}v_0\|_{B^{s}_{p,r}}\\
&
+C(\|u_0\|_{B^{s}_{p,r}}+\|v_0\|_{B^{s}_{p,r}})^2\int_0^{t}\frac{1}{(1-2Ct(\|u_0\|_{B^{s}_{p,r}}+\|v_0\|_{B^{s}_{p,r}}))^{\frac{3}{2}}}dt'\big)\\
&
\leq \big(1-2C(\|u_0\|_{B^{s}_{p,r}}+\|v_0\|_{B^{s}_{p,r}})t\big)^{-\frac{1}{2}}\big(\|S_{n+1}u_0\|_{B^{s}_{p,r}}+\|S_{n+1}v_0\|_{B^{s}_{p,r}}\\
&
\times(1-2Ct(\|u_0\|_{B^{s}_{p,r}}+\|v_0\|_{B^{s}_{p,r}}))^{-\frac{1}{2}}\big)\\
&
\leq\frac{\|u_0\|_{B^{s}_{p,r}}+\|v_0\|_{B^{s}_{p,r}}}{1-2Ct(\|u_0\|_{B^{s}_{p,r}}+\|v_0\|_{B^{s}_{p,r}})}.
\end{align*}
Therefore, $(u^n ,v^n)_{n\in \mathbb{N}}$ is bounded in $L^{\infty}([0,T], B_{p,r}^{s} \times B_{p,r}^{s})$.

\textbf{Step 3:} Cauchy sequence

Next we are going to prove that $(u^n, v^n)_{n\in \mathbb{N}}$ is a Cauchy sequence in $L^{\infty}([0, T]; B^{s-1}_{p, r}\times B^{s-1}_{p, r})$. For that purpose, we have, for all $(n, m) \in \mathbb{N}^2$,
$$u_{t}^{m+n+1}+G^{m+n}u_x^{m+n+1}=p*\big(-3(2u_x^{m+n}+v_x^{m+n})u+v_x^{m+n}u_{xx}^{m+n}-v_{xx}^{m+n}u_{x}^{m+n}\big)\overset{\Delta}{=}F^{m+n},$$
$$u_{t}^{n+1}+G^{n}u_x^{n+1}=p*\big(-3(2u_x^{n}+v_x^{n})u+v_x^{n}u_{xx}^{n}-v_{xx}^{n}u_{x}^{n}\big)\overset{\Delta}{=}F^{n},$$
then, we have
$$u_{t}^{m+n+1}-u_{t}^{n+1}+G^{m+n}(u_{x}^{m+n+1}-u_{x}^{n+1})=(G^{n}-G^{m+n})u_{x}^{n+1}+F^{m+n}-F^{n},$$
$$v_{t}^{m+n+1}-v_{t}^{n+1}+G^{m+n}(v_{x}^{m+n+1}-v_{x}^{n+1})=(G^{n}-G^{m+n})v_{x}^{n+1}+H^{m+n}-H^{n},$$
where
\begin{align*}
G^n-G^{m+n}&=2(u^n-u^{n+m})+v^n-v^{n+m},
\\
F^{m+n}-F^{n}&=p*(-6u_{x}^{m+n}u^{m+n}-3v_{x}^{m+n}u^{m+n}+v_{x}^{m+n}u_{xx}^{m+n}-v_{xx}^{m+n}u_{x}^{m+n}\\
&+6u_{x}^{n}u^{n}+3v_{x}^{n}u^{n}-v_{x}^{n}u_{xx}^{n}+v_{xx}^{n}u_{x}^{n}),
\\
H^{m+n}-H^{n}&=p*(-4u_{x}^{m+n}v^{m+n}-2v_{x}^{m+n}v^{m+n}-2v_{x}^{m+n}u_{xx}^{m+n}-v_{xx}^{m+n}v_{x}^{m+n}\\
&+4u_{x}^{n}v^{n}+2v_{x}^{n}v^{n}+2v_{x}^{n}u_{xx}^{n}+v_{x}^{n}v_{xx}^{n}).
\end{align*}
Applying Lemma 2.8 yields, for any $t\in [0, T],$
\begin{align}\label{eq316}
\|(u^{m+n+1}-u^{n+1})(t)\|_{B^{s-1}_{p,r}}&\leq e^{C\int_0^t \|G_x^{m+n}\|_{B^{s-1}_{p,r}}dt'}\Big(\|S_{m+n+1}u_0-S_{n+1}u_0\|_{B^{s-1}_{p,r}}\notag\\
&
+\int_0^t e^{-C\int_0^{t'} \|G_x^{m+n}\|_{B^{s-1}_{p,r}} dt''}(\|(G^n-G^{m+n})u_x^{n+1}\|_{B^{s-1}_{p,r}}+\|F^{m+n}-F^{n}\|_{B^{s-1}_{p,r}}dt'\Big)
\end{align}
and
\begin{align}\label{eq317}
\|(v^{m+n+1}-v^{n+1})(t)\|_{B^{s-1}_{p,r}}&\leq e^{C\int_0^t \|G_x^{m+n}\|_{B^{s-1}_{p,r}}dt'}\Big(\|S_{m+n+1}v_0-S_{n+1}v_0\|_{B^{s-1}_{p,r}}\notag\\
&
+\int_0^t e^{-C\int_0^{t'} \|G_x^{m+n}\|_{B^{s-1}_{p,r}} dt''}(\|(G^n-G^{m+n})v_x^{n+1}\|_{B^{s-1}_{p,r}}+\|H^{m+n}-H^{n}\|_{B^{s-1}_{p,r}}dt'\Big).
\end{align}
For $s>\frac{1}{p}+\frac{3}{2}$, $B_{p, r}^{s-1} \hookrightarrow L^{\infty}$ is an algebra, according Proposition 2.3 and Lemma 2.5, we have
\begin{align}\label{eq318}
\|(G^n-G^{m+n})u_{x}^{n+1}\|_{B^{s-1}_{p,r}}&\leq \|G^n-G^{m+n}\|_{B^{s-1}_{p,r}}\|u_x^{n+1}\|_{B^{s-1}_{p,r}}\notag\\
&
\leq C(\|u^n-u^{m+n}\|_{B^{s-1}_{p,r}}+\|v^n-v^{m+n}\|_{B^{s-1}_{p,r}})\|u^{n+1}\|_{B^{s}_{p,r}}.
\end{align}
According to Proposition 2.3, Lemmas 2.5 and 3.2, and the assumption about $s$ in Theorem \ref{theorem}, we obtain
\begin{align*}
\|p*(-6u_{x}^{m+n}u^{m+n}+6u_{x}^{n}u^{n})\|_{B^{s-1}_{p,r}}&\leq C\|-6u_{x}^{m+n}(u^{m+n}-u^n)-6u^{n}(u_x^{m+n}-u_x^{n})\|_{B^{s-3}_{p,r}}\\
&
\leq C\|u_x^{m+n}\|_{B^{s-1}_{p,r}}\|u^{m+n}-u^{n}\|_{B^{s-2}_{p,r}}+C\|u^{n}\|_{B^{s}_{p,r}}\|u_x^{m+n}-u_x^{n}\|_{B^{s-2}_{p,r}}\\
&
\leq C\|u^{m+n}\|_{B^{s}_{p,r}}\|u^{m+n}-u^{n}\|_{B^{s-1}_{p,r}}+C\|u^{n}\|_{B^{s}_{p,r}}\|u^{m+n}-u^{n}\|_{B^{s-1}_{p,r}},
\\
\|p*(-3v_{x}^{m+n}u^{m+n}+3v_{x}^{n}u^{n})\|_{B^{s-1}_{p,r}}&\leq C\|-3v_{x}^{m+n}(u^{m+n}-u^n)-3u^{n}(v_x^{m+n}-v_x^{n})\|_{B^{s-3}_{p,r}}\\
&
\leq C\|v_x^{m+n}\|_{B^{s-3}_{p,r}}\|u^{m+n}-u^{n}\|_{B^{s-2}_{p,r}}+c\|u^{n}\|_{B^{s}_{p,r}}\|v_x^{m+n}-v_x^{n}\|_{B^{s-2}_{p,r}}\\
&
\leq C\|v^{m+n}\|_{B^{s}_{p,r}}\|u^{m+n}-u^{n}\|_{B^{s-1}_{p,r}}+c\|u^{n}\|_{B^{s}_{p,r}}\|v^{m+n}-v^{n}\|_{B^{s-1}_{p,r}},
\\
\|p*(v_{x}^{m+n}u_{xx}^{m+n}-v_{x}^{n}u_{xx}^{n})\|_{B^{s-1}_{p,r}}&\leq C\|v_{x}^{m+n}(u_{xx}^{m+n}-u_{xx}^n)+u_{xx}^{n}(v_x^{m+n}-v_x^{n})\|_{B^{s-3}_{p,r}}\\
&
\leq C\|v_x^{m+n}\|_{B^{s-1}_{p,r}}\|u_{xx}^{m+n}-u_{xx}^{n}\|_{B^{s-3}_{p,r}}+C\|u_{xx}^{n}\|_{B^{s-2}_{p,r}}\|v_x^{m+n}-v_x^{n}\|_{B^{s-2}_{p,r}}\\
&
\leq C\|v^{m+n}\|_{B^{s}_{p,r}}\|u^{m+n}-u^{n}\|_{B^{s-1}_{p,r}}+C\|u^{n}\|_{B^{s}_{p,r}}\|v^{m+n}-v^{n}\|_{B^{s-1}_{p,r}},
\\
\|p*(-v_{xx}^{m+n}u_{x}^{m+n}+v_{xx}^{n}u_{x}^{n})\|_{B^{s-1}_{p,r}}&\leq C\|-v_{xx}^{m+n}(u_{x}^{m+n}-u_{x}^n)-u_{x}^{n}(v_{xx}^{m+n}-v_{xx}^{n})\|_{B^{s-3}_{p,r}}\\
&
\leq C\|v_{xx}^{m+n}\|_{B^{s-2}_{p,r}}\|u_{x}^{m+n}-u_{x}^{n}\|_{B^{s-2}_{p,r}}+C\|u_{x}^{n}\|_{B^{s-1}_{p,r}}\|v_{xx}^{m+n}-v_{xx}^{n}\|_{B^{s-3}_{p,r}}\\
&
\leq C\|v^{m+n}\|_{B^{s}_{p,r}}\|u^{m+n}-u^{n}\|_{B^{s-1}_{p,r}}+C\|u^{n}\|_{B^{s}_{p,r}}\|v^{m+n}-v^{n}\|_{B^{s-1}_{p,r}}.
\end{align*}
According to the above inequalities, we get
\begin{align}\label{eq319}
\|F^{m+n}-F^{n}\|_{B^{s-1}_{p,r}}&\leq C(\|u^{m+n}\|_{B^{s}_{p,r}}+\|v^{m+n}\|_{B^{s}_{p,r}}+\|u^{n}\|_{B^{s}_{p,r}})\|u^{m+n}-u^{n}\|_{B^{s-1}_{p,r}}\notag\\
&+C\|u^n\|_{B^{s}_{p,r}}\|v^{m+n}-v^{n}\|_{B^{s-1}_{p,r}}.
\end{align}
Similar to the calculation of $F^{n}$, we get
\begin{align}\label{eq320}
\|H^{m+n}-H^{n}\|_{B^{s-1}_{p,r}}&\leq C(\|u^{m+n}\|_{B^{s}_{p,r}}+\|v^{m+n}\|_{B^{s}_{p,r}}+\|u^{n}\|_{B^{s}_{p,r}}+\|v^{n}\|_{B^{s}_{p,r}})\|v^{m+n}-v^{n}\|_{B^{s-1}_{p,r}}\notag\\
&+C(\|u^n\|_{B^{s}_{p,r}}+\|u^{m+n}\|_{B^{s}_{p,r}})\|u^{m+n}-u^{n}\|_{B^{s-1}_{p,r}}.
\end{align}
Since $(u^n, v^n)_{n\in \mathbb{N}}$ and $(u^{m+n}, v^{m+n})_{m+n\in \mathbb{N}}$ are bounded in $L^{\infty}([0, T]; B^s_{p, r}\times B^s_{p, r})$ for all in $[0, T].$
The combination \eqref{eq316}, \eqref{eq317}, \eqref{eq318}, \eqref{eq319} and \eqref{eq320} yields
\begin{align*}
\|(u^{m+n+1}-u^{n+1})(t)\|_{B^{s-1}_{p,r}}&\leq C\Big(\|S_{m+n+1}u_0-S_{n+1}u_0\|_{B^{s-1}_{p,r}}+\int_0^t(\|u^{m+n}-u^{n}\|_{B^{s-1}_{p,r}}+\|v^{m+n}-v^{n}\|_{B^{s-1}_{p,r}})dt'\Big),
\\
\|(v^{m+n+1}-v^{n+1})(t)\|_{B^{s-1}_{p,r}}&\leq C\Big(\|S_{m+n+1}v_0-S_{n+1}v_0\|_{B^{s-1}_{p,r}}+\int_0^t(\|u^{m+n}-u^{n}\|_{B^{s-1}_{p,r}}+\|v^{m+n}-v^{n}\|_{B^{s-1}_{p,r}})dt'\Big).
\end{align*}
Hence, we have
\begin{align*}
\|(u^{m+n+1}-u^{n+1})(t)\|_{B^{s-1}_{p,r}}&+\|(v^{m+n+1}-v^{n+1})(t)\|_{B^{s-1}_{p,r}}\leq C\Big(\|S_{m+n+1}u_0-S_{n+1}u_0\|_{B^{s-1}_{p,r}}\\
&
+\|S_{m+n+1}v_0-S_{n+1}v_0\|_{B^{s-1}_{p,r}}+\int_0^t(\|u^{m+n}-u^{n}\|_{B^{s-1}_{p,r}}+\|v^{m+n}-v^{n}\|_{B^{s-1}_{p,r}})dt'\Big).
\end{align*}
Taking an upper bound on $[0, t]$ yields
\begin{align}\label{eq321}
\|u^{m+n+1}&-u^{n+1}\|_{L^\infty_{t}({B^{s-1}_{p,r}})}+\|v^{m+n+1}-v^{n+1}\|_{L^\infty_{t}({B^{s-1}_{p,r}})}\leq C\Big(\|S_{m+n+1}u_0-S_{n+1}u_0\|_{B^{s-1}_{p,r}}\notag\\
&
+\|S_{m+n+1}v_0-S_{n+1}v_0\|_{B^{s-1}_{p,r}}+\int_0^t(\|u^{m+n}-u^{n}\|_{L^\infty_{t'}({B^{s-1}_{p,r}})}+\|v^{m+n}-v^{n}\|_{L^\infty_{t'}({B^{s-1}_{p,r}})})dt'\Big).
\end{align}
Let $g_n(t)=\sup_m(\|u^{m+n}-u^{n}\|_{L^{\infty}_{t}(B^{s-1}_{p,r})}+\|v^{m+n}-v^{n}\|_{L^{\infty}_{t}(B^{s-1}_{p,r})})$. Then \eqref{eq321} can be written as
$$g_{n+1}(t)\leq C\sup_m(\|S_{m+n+1}u_0-S_{n+1}u_0\|_{B^{s-1}_{p,r}}+\|S_{m+n+1}v_0-S_{n+1}v_0\|_{B^{s-1}_{p,r}}+\int_0^{t}g_n(t')dt').$$
Obviously, $(S_nu^0, S_nv^0)_{n\in \mathbb{N}}$ is a Cauchy sequence in $B^{s-1}_{p, r}$. By Fatou's lemma, we have
$$
g(t)\overset{\Delta}{=}\limsup_{n\rightarrow\infty}g_{n+1}(t)\leq C\int_0^tg(t')dt'.
$$
The Gronwall inequality entails that $g(t)=0$ for all $t\in [0,T]$. Therefore $(u^n, v^n)_{n\in \mathbb{N}}$ is a Cauchy sequence in $C([0,T]; B^{s-1}_{p,r}\times B^{s-1}_{p,r})$ and converges to some limit function $(u, v) \in C([0,T]; B^{s-1}_{p,r}\times B^{s-1}_{p,r})$.

\textbf{Step 4:} Convergence

We have to prove that $(u, v)$ belongs to $E^s_{p,r}(T)\times E^s_{p,r}(T)$ and solves \eqref{eq31}. Since $(u^n, v^n)_{n\in \mathbb{N}}$ is uniformly bounded in $L^{\infty}([0,T]; B^{s}_{p,r}\times B^{s}_{p,r})$,
we can use the Fatou property for the Besov spaces to show that $(u, v)$ also belongs to $L^{\infty}([0,T];B^{s}_{p,r}\times B^{s}_{p,r})$.
On the other hand, as $(u^n, v^n)_{n\in \mathbb{N}}$ converges to $(u, v)$ in $C([0,T]; B^{s-1}_{p,r}\times B^{s-1}_{p,r})$,
an interpolation argument insure that the convergence holds in $C([0,T];B^{s'}_{p,r}\times B^{s'}_{p,r})$ for any $s'<s$.
Then it is easy to pass to the limit in \eqref{eq39} and to conclude that $(u, v)$ is indeed a solution of \eqref{eq31} in the sense of distributions.

Finally, thanks to the fact that $(u, v)$ belongs to $L^{\infty}([0,T]; B^{s}_{p,r}\times B^{s}_{p,r})$, we known that the right-hand side of \eqref{eq31}
also belongs to $L^{\infty}([0,T]; B^{s}_{p,r}\times B^{s}_{p,r})$.
Taking advantage of Lemma \ref{existence}, we can deduce that $(u, v)$ belongs to $C([0,T];B^{s}_{p,r}\times B^{s}_{p,r})$ (resp., $C_w([0,T]; B^s_{p,r}\times B^{s}_{p,r})$)
if $r<\infty$ (resp., $r=\infty$). Again using the equation \eqref{eq31}, we see that $(u_t, v_t)$ is in $C([0,T];B^{s-1}_{p,r}\times B^{s-1}_{p,r})$ if $r$ is finite,
and in $L^{\infty}([0,T]; B^{s-1}_{p,r}\times B^{s-1}_{p,r})$ otherwise. All in all, $(u, v)$ belongs to $E^s_{p,r}(T)\times E^s_{p,r}(T)$.

Secondly, we are going to prove the uniqueness of solutions to the equation \eqref{eq31} based on the proof of the third step.

Suppose $(u^1, v^1)$ and $ (u^2, v^2)$ are two sets of solutions
of the system (3.1) with the initial data $(u_0^1, v_0^1)$ and $ (u_0^2, v_0^2)$, respectively. Denoting $u^{12}\overset{\Delta}{=}u^1-u^2$ and $v^{12}\overset{\Delta}{=}v^1-v^2$, We obtain
\begin{equation}\label{eq322}
\left\{
  \begin{array}{l}
u^{12}_t+G^1u^{12}_x=-(G^1-G^2)u^{2}_x+F^1-F^2,
\\
v^{12}_t+G^1u^{12}_x=-(G^1-G^2)v^{2}_x+H^1-H^2,
\\
u^{12}|_{t=0}=u^{12}_0\overset{\Delta}{=}u^1_0-u^2_0, v^{12}|_{t=0}=v^{12}_0\overset{\Delta}{=}v^1_0-v^2_0,
  \end{array}
  \right.
\end{equation}
where for $i=1,2,$
\begin{align*}
G^i=2u^i+v^i, \quad F^i=p*(-3(2u_x^i+v_x^i)u^i+v_x^iu_{xx}^i-v_{xx}^iu_{x}^i),\\
H^i=p*(-2(2u_x^i+v_x^i)v^i+2v_x^iu_{xx}^i+v_{xx}^iv_{x}^i).
\end{align*}
By virtue to Lemma \ref{priori estimate}, we have
\begin{equation}\label{eq323}
\begin{split}
\|u^{12}(t)\|_{B^{s-1}_{p,r}}&\leq e^{C\int_0^t \|G_x^{1}\|_{B^{s-1}_{p,r}}dt'}\Big(\|(u^1-u^2)(0)\|_{B^{s-1}_{p,r}}\\
&
+\int_0^t e^{-C\int_0^{t'} \|G_x^{1}\|_{B^{s-1}_{p,r}} dt''}(\|(G^2-G^{1})u_x^{2}\|_{B^{s-1}_{p,r}}+\|F^{1}-F^{2}\|_{B^{s-1}_{p,r}}dt'\Big),
\\
\|v^{12}(t)\|_{B^{s-1}_{p,r}}&\leq e^{C\int_0^t \|G_x^{1}\|_{B^{s-1}_{p,r}}dt'}\Big(\|(v^1-v^2)(0)\|_{B^{s-1}_{p,r}}\\
&
+\int_0^t e^{-C\int_0^{t'} \|G_x^{1}\|_{B^{s-1}_{p,r}} dt''}(\|(G^2-G^{1})v_x^{2}\|_{B^{s-1}_{p,r}}+\|H^{1}-H^{2}\|_{B^{s-1}_{p,r}}dt'\Big).
\end{split}
\end{equation}
Similar to the calculation in Step 3, we get
\begin{equation}\label{eq324}
\begin{split}
&\|(G^2-G^{1})u_{x}^{2}\|_{B^{s-1}_{p,r}}\leq C(\|u^2-u^{1}\|_{B^{s-1}_{p,r}}+\|v^2-v^{1}\|_{B^{s-1}_{p,r}})\|u^{2}\|_{B^{s}_{p,r}},
\\
&\|F^{1}-F^{2}\|_{B^{s-1}_{p,r}}\leq C(\|u^{1}\|_{B^{s}_{p,r}}+\|v^{1}\|_{B^{s}_{p,r}}+\|u^{2}\|_{B^{s}_{p,r}}+\|v^{2}\|_{B^{s}_{p,r}})(\|u^{12}\|_{B^{s-1}_{p,r}}+\|v^{12}\|_{B^{s-1}_{p,r}}),
\\
&\|H^{1}-H^{2}\|_{B^{s-1}_{p,r}}\leq C(\|u^{1}\|_{B^{s}_{p,r}}+\|v^{1}\|_{B^{s}_{p,r}}+\|u^{2}\|_{B^{s}_{p,r}}+\|v^{2}\|_{B^{s}_{p,r}})(\|u^{12}\|_{B^{s-1}_{p,r}}+\|v^{12}\|_{B^{s-1}_{p,r}}).
\end{split}
\end{equation}
Substituting \eqref{eq324} into \eqref{eq323} yields that
\begin{align*}
\|u^{12}(t)\|_{B^{s-1}_{p,r}}+\|v^{12}(t)\|_{B^{s-1}_{p,r}}
\leq &e^{C\int_0^t (\|u^1\|_{B^{s}_{p,r}}+\|v^1\|_{B^{s}_{p,r}})dt'}\Big(\|u_0^{12}\|_{B^{s-1}_{p,r}}+\|v_0^{12}\|_{B^{s-1}_{p,r}}\\
&
+C\int_0^t e^{-C\int_0^{t'} (\|u^1\|_{B^{s}_{p,r}}+\|v^1\|_{B^{s}_{p,r}}) dt''}
(\|u^{1}\|_{B^{s}_{p,r}}+\|v^{1}\|_{B^{s}_{p,r}}+\|u^{2}\|_{B^{s}_{p,r}}\\
&
+\|v^{2}\|_{B^{s}_{p,r}})(\|u^{12}\|_{B^{s-1}_{p,r}}+\|v^{12}\|_{B^{s-1}_{p,r}})dt'\Big).
\end{align*}
Hence, applying Gronwall's inequality, we can get
\begin{equation}\label{eq3251}
\begin{split}
\|u^{1}(t)-u^{2}(t)\|_{B^{s-1}_{p,r}}+\|v^{1}(t)-v^{2}(t)\|_{B^{s-1}_{p,r}}&\leq (\|u_0^{1}-u_0^{2}\|_{B^{s-1}_{p,r}}+\|v_0^{1}-v_0^{2}\|_{B^{s-1}_{p,r}})\\
&
\times e^{C\int_0^{t} (\|u^{1}\|_{B^{s}_{p,r}}+\|v^{1}\|_{B^{s}_{p,r}}+\|u^{2}\|_{B^{s}_{p,r}}+\|v^{2}\|_{B^{s}_{p,r}}) dt'}.
\end{split}
\end{equation}
If choosing the initial value $(u_0^1, v_0^1)=(u_0^2, v_0^2)$, we can get $(u^1, v^1)=(u^2, v^2)$  from equation \eqref{eq3251}.
Therefore, the uniqueness of the solution of system \eqref{eq31} in $E^s_{p,r}(T)\times E^s_{p,r}(T)$ space is proved.

Combing with the proofs of the existence and uniqueness of the solution, we completes the proof of Theorem \ref{theorem}.

\textbf{Remark 1.} The Theorem \ref{theorem} improves the corresponding result in \cite{3}, where $s>\max\{\frac{5}{2}, 2+\frac{1}{p}\}$.

\begin{theo}\label{theorem34}
Let $(s,p,r)$ be the statement of Theorem \ref{theorem}. Denote $\bar{\mathbb{N}}=\mathbb{N}\cup\{\infty\}$. Suppose that $(u^n, v^n)_{n\in\bar{\mathbb{N}}}$
is the corresponding solution to \eqref{eq31} given by Theorem \ref{theorem} with the initial data $(u_0^n, v_0^n)\in B^{s-1}_{p,r} \times B^{s-1}_{p,r}$. If $(u_0^n,v_0^n) \rightarrow (u_0^{\infty}, v_0^{\infty})$ in $B^{s-1}_{p,r}\times B^{s-1}_{p,r}$, then $(u^n, v^n) \rightarrow (u^{\infty},v^{\infty})$ in $C([0,T];B^{s-1}_{p,r}\times B^{s-1}_{p,r})\ (resp., C_w([0,T];B^{s-1}_{p,r}\times B^{s-1}_{p,r}))$ if $r<\infty\ (resp., r=\infty)$ with
$2C T\sup_{n\in\bar{\mathbb{N}}}(\|u^n_0\|_{B^{s}_{p,r}}+\|u^n_0\|_{B^{s}_{p,r}})<1.$
\end{theo}

\begin{proof}
According to the proof of the existence, we find for all $n\in\bar{\mathbb{N}},\ t\in[0,T]$,
$$\|u^n(t)\|_{B^s_{p,r}}+\|v^n(t)\|_{B^s_{p,r}}\leq \frac{C(\|u^n_0\|_{B^{s}_{p,r}}+\|v^n_0\|_{B^{s}_{p,r}})}
{1-2C t(\|u^n_0\|_{B^{s}_{p,r}}+\|v^n_0\|_{B^{s}_{p,r}})}.
$$
Then $(u^n, v^n)_{n\in\bar{\mathbb{N}}}$ is bounded in $L^{\infty}([0,T]; B^s_{p,r}\times B^s_{p,r})$.
As in the proof of the above theorem, we can decompose $u^n$ into $u^n=y^n+z^n$ and $v^n$ into $v^n=g^n+f^n$ such that
\begin{equation*}
    \left\{\begin{array}{l}
    y_t^n+G^ny_x^n=F^{\infty},
    \\
    g_t^n+G^ng_x^n=H^{\infty},
    \\
    y^n|_{t=0}=u_0^{\infty},
    \\
     g^n|_{t=0}=v_0^{\infty},
    \end{array}\right.
  \text{and}\quad
    \left\{\begin{array}{l}
    z_t^n+G^nz_x^n=F^n-F^{\infty},
     \\
     f_t^n+G^nf_x^n=H^n-H^{\infty},
     \\
     z^n|_{t=0}=u_0^n-u_0^{\infty},
     \\
      f^n|_{t=0}=v_0^n-v_0^{\infty}.
    \end{array}\right.
\end{equation*}
Obviously we have
\begin{equation}\label{eq326}
\begin{split}
\|G^n\|_{B^{s}_{p,r}}&=\|2u^n+v^n\|_{B^{s}_{p,r}}\leq C(\|u^n\|_{B^s_{p,r}}+\|v^n\|_{B^s_{p,r}})
\\
\|G^n-G^{\infty}\|_{B^{s-1}_{p,r}}&=\|2(u^n-u^{\infty})+(v^n-v^{\infty})\|_{B^{s-1}_{p,r}}\leq C(\|u^n-u^{\infty}\|_{B^{s-1}_{p,r}}+\|v^n-v^{\infty}\|_{B^{s-1}_{p,r}}.
\end{split}
\end{equation}
We have already known $(u^n,v^n) \rightarrow (u^{\infty}, v^{\infty})$ in $L^{\infty}([0,T]; B^{s-1}_{p,r}\times B^{s-1}_{p,r})$.
At the same time, by \eqref{eq326},
$G^n$ satisfy the condition of Lemma \ref{continuity}. This leads to that $(y^n, g^n)\rightarrow (y^{\infty}, g^{\infty})$ in $C([0,T];B^{s-1}_{p,r}\times B^{s-1}_{p,r})$ if $r<\infty$.

According to Lemma \ref{priori estimate}, we have for all $n\in\mathbb{N}$ and $t\in[0,T]$,
\begin{equation}\label{eq327}
\begin{split}
\|z^n(t)\|_{B^{s-1}_{p,r}}
\leq e^{C\int_0^t \|G_x^n\|_{B^{s-1}_{p,r}} dt'}\Big(\|u_0^n-u_0^{\infty}\|_{B^{s-1}_{p,r}}+\int_0^t e^{-C\int_0^{t'} \|G_x^n\|_{B^{s-1}_{p,r}} dt''}
\|F^n-F^{\infty}\|_{B^{s-1}_{p,r}}dt'\Big),
\\
\|f^n(t)\|_{B^{s-1}_{p,r}}
\leq e^{C\int_0^t \|G_x^n\|_{B^{s-1}_{p,r}} dt'}\Big(\|v_0^n-v_0^{\infty}\|_{B^{s-1}_{p,r}}+\int_0^t e^{-C\int_0^{t'} \|G_x^n\|_{B^{s-1}_{p,r}} dt''}
\|H^n-H^{\infty}\|_{B^{s-1}_{p,r}}dt'\Big).
\end{split}
\end{equation}
Similar to the calculations in equations \eqref{eq319} and \eqref{eq320}, we can easily deduce that
\begin{equation}\label{eq328}
\begin{split}
\|G_x^n\|_{B^{s-1}_{p,r}}&=\|2u_x^n+v_x^n\|_{B^{s-1}_{p,r}}\leq C(\|u^n\|_{B^s_{p,r}}+\|v^n\|_{B^s_{p,r}}),
\\
\|F^{n}-F^{\infty}\|_{B^{s-1}_{p,r}}&\leq C(\|u^{n}\|_{B^{s}_{p,r}}+\|v^{n}\|_{B^{s}_{p,r}}+\|u^{\infty}\|_{B^{s}_{p,r}}+\|v^{\infty}\|_{B^{s}_{p,r}})
\\
&\times(\|u^{n}-u^{\infty}\|_{B^{s-1}_{p,r}}+\|v^{n}-v^{\infty}\|_{B^{s-1}_{p,r}}),
\\
\|H^{n}-H^{\infty}\|_{B^{s-1}_{p,r}}&\leq C(\|u^{n}\|_{B^{s}_{p,r}}+\|v^{n}\|_{B^{s}_{p,r}}+\|u^{\infty}\|_{B^{s}_{p,r}}+\|v^{\infty}\|_{B^{s}_{p,r}})
\\
&\times(\|u^{n}-u^{\infty}\|_{B^{s-1}_{p,r}}+\|v^{n}-v^{\infty}\|_{B^{s-1}_{p,r}}).
\end{split}
\end{equation}
Plugging \eqref{eq328} into \eqref{eq327}, and using the boundedness of $(u^n, u^n)$ and $(u^{\infty}, u^{\infty})$ in space $C([0,T];B^{s}_{p,r}\times B^{s}_{p,r})$, we obtain
\begin{equation*}
\begin{split}
\|z^n(t)\|_{B^{s-1}_{p,r}}+\|f^n(t)\|_{B^{s-1}_{p,r}}\leq &C\big(\|u_0^{n}-u_0^{\infty}\|_{B^{s-1}_{p,r}}+\|v_0^{n}-v_0^{\infty}\|_{B^{s-1}_{p,r}}
\\
&+\int_0^{t}(\|u^{n}-u^{\infty}\|_{B^{s-1}_{p,r}}+\|v^{n}-v^{\infty}\|_{B^{s-1}_{p,r}})dt'\big).
\end{split}
\end{equation*}
Noticing that $(y^{\infty},g^{\infty})=(u^{\infty}, v^{\infty})$ and $(z^{\infty}, f^{\infty})=(0, 0)$, we have
\begin{equation*}\label{eq325}
\begin{split}
\|z^n(t)\|_{B^{s-1}_{p,r}}&+\|f^n(t)\|_{B^{s-1}_{p,r}}\leq C\big(\|u_0^{n}-u_0^{\infty}\|_{B^{s-1}_{p,r}}+\|v_0^{n}-v_0^{\infty}\|_{B^{s-1}_{p,r}}
\\
&+\int_0^{t}(\|y^{n}-y^{\infty}\|_{B^{s-1}_{p,r}}+\|g^{n}-g^{\infty}\|_{B^{s-1}_{p,r}}+\|z^n\|_{B^{s-1}_{p,r}}+\|f^n\|_{B^{s-1}_{p,r}})dt'\big).
\end{split}
\end{equation*}
Appling Gronwall's inequality yields that
\begin{equation*}
   \|z^n(t)\|_{B^{s-1}_{p,r}}+\|f^n(t)\|_{B^{s-1}_{p,r}}
    \leq e^{Ct}\Big(\|u_0^{n}-u_0^{\infty}\|_{B^{s-1}_{p,r}}+\|v_0^{n}-v_0^{\infty}\|_{B^{s-1}_{p,r}}+\int_0^t e^{-Ct}\|y^{n}-y^{\infty}\|_{B^{s-1}_{p,r}}+\|g^{n}-g^{\infty}\|_{B^{s-1}_{p,r}}dt'\Big).
\end{equation*}
Therefore when $r<\infty,\ (z^n, f^n)\rightarrow (0,0)$ in $C([0,T];B^{s-1}_{p,r}\times B^{s-1}_{p,r})$, and thus $(u^n, v^n)\rightarrow (u^{\infty}, v^{\infty})$ in $C([0,T]; B^{s-1}_{p,r}\times B^{s-1}_{p,r})$.

As for the case $r=\infty$, we have weak continuity. In fact,
for fixed $\phi\in B^{-(s-1)}_{p',1}$, we write
$$\langle u^n(t)-u^{\infty}(t),\phi\rangle=\langle u^n(t)-u^{\infty}(t),S_j\phi\rangle+\langle u^n(t)-u^{\infty}(t),\phi-S_j\phi\rangle.$$
By duality, we have
\begin{equation*}
    |\langle u^n(t)-u^{\infty}(t),\phi\rangle|\leq \|u^n(t)-u^{\infty}(t)\|_{B^{s-2}_{p,\infty}}\|S_j\phi\|_{B^{2-s}_{p',1}}+\|u^n(t)-u^{\infty}(t)\|_{B^{s-1}_{p,\infty}}\|\phi-S_j\phi\|_{B^{-(s-1)}_{p',1}}.
\end{equation*}
Using the fact that $u^n\rightarrow u^{\infty}$ in $L^{\infty}([0,T];B^{s-2}_{p,\infty})$, and $S_j \phi\rightarrow \phi$ in $B^{-(s-1)}_{p',1}$ and $(u^n)_{n\in\bar{\mathbb{N}}}$ is bounded in $L^{\infty}([0,T];B^{s-1}_{p,r})$, it is now simple to conclude that $\langle u^n(t)-u^{\infty}(t),\phi\rangle\rightarrow 0$ uniformly on $[0,T]$.
For the function $v$, the same result is obtained by calculation similar to $u$.
\end{proof}

\section{Blow-up criterion and global existence}

\subsection{Blow-up criterion}

In the section, we will establish a global existence result for the Cauchy problem of \eqref{eq1} in Besov spaces. First, we state a blow-up criterion for \eqref{eq31}. In what follows, we let $T^{*}$
be the lifespan of thesolution of \eqref{eq1} with the initial data $(u_0, v_0)$ as the supremum of positive times $T$ such that \eqref{eq1} has a solution $(u, v)\in E^s_{p,r} \times E^s_{p,r}$
on $[0, T]\times \mathbb{R}$. We have the following result.

\begin{theo}\label{theoxx}
Let $u_{0}, v_{0}\in B^{s}_{p,r} \times B^{s}_{p,r}$ with $(s,p,r)$  satisfies $s> \max\{2, \frac{1}{p}+\frac{3}{2}\}$ or $(s=2, p=2, r=2)$, and let $T^{*}$ be the maximal existence
time of the corresponding solution $(u,v)$ to (3.1). Then $(u,v)$ blows up in finite time $T^{*}$ if and only if
$$\int_{0}^{T^{*}}(\|u_x\|_{L^\infty}+\|v_x\|_{L^\infty}+\|u\|_{L^\infty}+\|v\|_{L^\infty})dt^{'}=\infty.$$
\end{theo}

\textbf{Proof:}
Firstly, when $s> \max\{2, \frac{1}{p}+\frac{3}{2}\}$ and $1\leq p, r\leq \infty$,  applying Lemma 2.9 with $p_1=p$, we have $p_2=\infty$ and
\begin{equation}\label{eq4.1}
\|u\|_{B^{s}_{p,r}}\leq \|u_{0}\|_{B^{s}_{p,r}}+C\int_{0}^{t}(\|2u_x+v_{x}\|_{L^\infty} \|u\|_{B^{s}_{p,r}}+\|u_x\|_{L^\infty} \|2u_x+v_{x}\|_{B^{s-1}_{p,r}}+\|F\|_{B^{s}_{p,r}})dt^{'},
\end{equation}
In \eqref{eq4.1}, we have
\begin{equation}
 \left\{
  \begin{array}{l}
\|2u_x+v_{x}\|_{L^\infty} \|u\|_{B^{s}_{p,r}}\leq C(\|u_x\|_{L^\infty}+\|v_x\|_{L^\infty} )\|u\|_{B^{s}_{p,r}},\\
\|u_x\|_{L^\infty} \|2u_x+v_{x}\|_{B^{s-1}_{p,r}}\leq C\|u_x\|_{L^\infty}(\|u\|_{B^{s}_{p,r}}+\|v\|_{B^{s}_{p,r}}).\\
  \end{array}
 \right.
\end{equation}
According to Properties 2.6 and 2.7, since $L^{\infty}\hookrightarrow B_{\infty,\infty}^{0}$, we can get
\begin{equation}\label{eq4.3}
 \left\{
  \begin{array}{l}
 \|T_{u_{x}}v_{xx}\|_{B^{s-2}_{p,r}}\leq C\|u_{x}\|_{L^{\infty}}\|v_{xx}\|_{B^{s-2}_{p,r}}\leq C\|u_{x}\|_{L^{\infty}}\|v\|_{B^{s}_{p,r}},
 \\[1ex]
 \|T_{v_{xx}}u_{x}\|_{B^{s-2}_{p,r}}\leq C\|v_{xx}\|_{B^{-1}_{\infty,\infty}}\|u_{x}\|_{B^{s-1}_{p,r}}\leq C\|v_{x}\|_{B^{0}_{\infty,\infty}}\|u\|_{B^{s}_{p,r}}
 \leq\|v_{x}\|_{L^{\infty}}\|u\|_{B^{s}_{p,r}},
 \\[1ex]
 \|R(u_{x},v_{xx})\|_{B^{s-2}_{p,r}}\leq C\|u_{x}\|_{B^{0}_{\infty,\infty}}\|v_{xx}\|_{B^{s-2}_{p,r}}\leq C\|u_{x}\|_{L^{\infty}}\|v\|_{B^{s}_{p,r}}.\\[1ex]
  \end{array}
 \right.
\end{equation}
Applying Bony decomposition and \eqref{eq4.3}, we have
\begin{equation}\label{eq4.4} \|u_{x}v_{xx}\|_{B^{s-2}_{p,r}}=\|T_{u_{x}}v_{xx}+T_{v_{xx}}u_{x}+R(u_{x},v_{xx})\|_{B^{s-2}_{p,r}}
 \leq \|u_x\|_{L^\infty} \|v\|_{B^{s}_{p,r}}+\|v_x\|_{L^\infty}\|u\|_{B^{s}_{p,r}}.
\end{equation}
Similarly, we get
$$
\|v_{x}u_{xx}\|_{B^{s-2}_{p,r}}\leq \|v_x\|_{L^\infty} \|u\|_{B^{s}_{p,r}}+\|u_x\|_{L^\infty}\|v\|_{B^{s}_{p,r}}.
$$
According to Lemma 2.5, we have
\begin{equation}
\begin{split}
\|-3(2u_x+v_x)u\|_{B^{s-2}_{p,r}}&\leq\|-3(2u_x+v_x)\|_{L^\infty} \|u\|_{B^{s-2}_{p,r}}+\|-3(2u_x+v_x)\|_{B^{s-2}_{p,r}} \|u\|_{L^\infty}
\\
&\leq(\|u_x\|_{L^\infty}+\|v_x\|_{L^\infty})\|u\|_{B^{s}_{p,r}}+(\|u\|_{B^{s}_{p,r}}+\|v\|_{B^{s}_{p,r}}) \|u\|_{L^\infty}.
\end{split}
\end{equation}
Hence, using inequalities (4.4) and (4.5) yields
\begin{equation}
\begin{split}
\|F\|_{B^{s}_{p,r}}&=\|-3(2u_x+v_x)u+v_xu_{xx}-v_{xx}u_{x}\|_{B^{s-2}_{p,r}}
\\
&\leq (\|u_x\|_{L^\infty}+\|v_x\|_{L^\infty} )\|u\|_{B^{s}_{p,r}}+(\|u\|_{B^{s}_{p,r}}+\|v\|_{B^{s}_{p,r}})\|u\|_{L^\infty}
\\
&+(\|u_x\|_{L^\infty}\|v\|_{B^{s}_{p,r}}+\|v_x\|_{L^\infty}\|u\|_{B^{s}_{p,r}})
\\
&\leq(\|u_x\|_{L^\infty}+\|v_x\|_{L^\infty} +\|u\|_{L^\infty})(\|u\|_{B^{s}_{p,r}}+\|v\|_{B^{s}_{p,r}}).
\end{split}
\end{equation}
Substituting (4.2) and (4.6) into (4.1), we have
\begin{equation}
\|u\|_{B^{s}_{p,r}}\leq \|u_{0}\|_{B^{s}_{p,r}}+C\int_{0}^{t}(\|u_x\|_{L^\infty}+\|v_x\|_{L^\infty} +\|u\|_{L^\infty})(\|u\|_{B^{s}_{p,r}}+\|v\|_{B^{s}_{p,r}})dt^{'}.
\end{equation}
Similarly,
\begin{equation}
\|v\|_{B^{s}_{p,r}}\leq \|v_{0}\|_{B^{s}_{p,r}}+C\int_{0}^{t}(\|u_x\|_{L^\infty}+\|v_x\|_{L^\infty} +\|v\|_{L^\infty})(\|u\|_{B^{s}_{p,r}}+\|v\|_{B^{s}_{p,r}})dt^{'}.
\end{equation}
Combining (4.7) and (4.8), we get
\begin{equation*}
\begin{split}
\|u\|_{B^{s}_{p,r}}+\|v\|_{B^{s}_{p,r}}&\leq (\|u_{0}\|_{B^{s}_{p,r}}+\|v_{0}\|_{B^{s}_{p,r}})
\\
&+C\int_{0}^{t}\Big((\|u_x\|_{L^\infty}+\|v_x\|_{L^\infty} +\|u\|_{L^\infty}+\|v\|_{L^\infty})(\|u\|_{B^{s}_{p,r}}+\|v\|_{B^{s}_{p,r}})\Big)dt^{'}.
\end{split}
\end{equation*}
Applying Gronwall's inequality yields that
\begin{equation}\label{eq4.10}
\begin{split}
\|u\|_{B^{s}_{p,r}}+\|v\|_{B^{s}_{p,r}}\leq (\|u_{0}\|_{B^{s}_{p,r}}+\|v_{0}\|_{B^{s}_{p,r}})e^{\int_{0}^{t}(\|u_x\|_{L^\infty}+\|v_x\|_{L^\infty} +\|u\|_{L^\infty}+\|v\|_{L^\infty})dt^{'}}.
\end{split}
\end{equation}

Secondly, when $s=2, p=2, r=2$, note that $B^2_{2,2}=H^2$, similar to the first proof, we get
\begin{equation*}
\|u\|_{H^{2}}\leq \|u_{0}\|_{H^{2}}+c\int_{0}^{t}(\|2u_x+v_{x}\|_{L^\infty} \|u\|_{H^{2}}+\|u_x\|_{L^\infty} \|2u_x+v_{x}\|_{H^{1}}+\|F\|_{H^{2}})dt^{'},
\end{equation*}
where
\begin{equation*}
\begin{split}
\|F\|_{H^{2}}&=\|-3(2u_x+v_x)u+v_xu_{xx}-v_{xx}u_{x}\|_{L^{2}}
\\
&\leq C(\|(2u_x+v_x)u\|_{L^2} +\|v_xu_{xx}\|_{L^{2}}+\|v_{xx}u_{x}\|_{L^2})
\\
&\leq C(\|2u_x+v_x\|_{L^2}\|u\|_{L^{\infty}}+\|v_x\|_{L^\infty}\|u_{xx}\|_{L^{2}}+\|u_x\|_{L^\infty}\|v_{xx}\|_{L^{2}} )
\\
&\leq C(\|u\|_{L^{\infty}}+\|u_x\|_{L^\infty}+\|v_x\|_{L^\infty})(\|u\|_{H^{2}}+\|v\|_{H^{2}}).
\end{split}
\end{equation*}
An inequality similar to \eqref{eq4.10} can be obtained by proofs similar to the firstly.

If $T^*$ is finite, and $\int_{0}^{T^{*}}(\|u_x\|_{L^\infty}+\|v_x\|_{L^\infty}+\|u\|_{L^\infty}+\|v\|_{L^\infty})dt^{'}<\infty$,
then $(u, v)\in L^{\infty}([0,T^*);B^s_{p,r}\times B^s_{p,r})$, which
contradicts the assumption that $T^*$ is the maximal existence time.

On the other hand, by Theorem \eqref{theorem} and the fact that $B^s_{p,r}\hookrightarrow C^{0,1}$,
if $\int_{0}^{T^{*}}(\|u_x\|_{L^\infty}+\|v_x\|_{L^\infty}+\|u\|_{L^\infty}+\|v\|_{L^\infty})dt^{'}=\infty$, then $(u, v)$ must blow up in finite time.

\textbf{Remark 2.} The blow-up criterion Theorem \ref{theoxx} obtained in this paper is an improvement of \cite{4} and \cite{5}.
 The conditions required for their results are that there exists $M>0$ such that $\|u\|_{L^\infty}+\|v\|_{L^\infty}+\|u_x\|_{L^\infty}+\|v_x\|_{L^\infty}+\|u_{xx}\|_{L^\infty}+\|v_{xx}\|_{L^\infty}\leq M$
 for $t\in [0, T)$.

\subsection{Global existence}
First we prove a conserved quantity for \eqref{eq1}.

\begin{theo}
If $m_0, n_0\in H^{s}(\mathbb{R}), s>\frac{1}{2},$ then as long as the solution $m(t,x), n(t,x)$  given by (1.1) exists, we have
\begin{equation}
\int_{\mathbb{R}} m+n dx^{'} =\int_{\mathbb{R}}m_0+n_0 dx^{'}.
\end{equation}
\end{theo}

\textbf{Proof:} Arguing by density, it suffices to consider the case where $u, v\in C_{0}^{\infty}(\mathbb{R})$.
Using integration by parts, we deduce that
\begin{equation}
\begin{split}
\frac{d}{dt}\int_{\mathbb{R}} m+n dx^{'}=\int_{\mathbb{R}} m_{t}+n_{t} dx^{'}&=\int_{\mathbb{R}}-(2u+v)(m_x+n_x)-(2u_x+v_x)(3m+2n)dx^{'}
\\
&=\int_{\mathbb{R}}-(2u_x+v_x)(2m+n)dx^{'}
\\
&=\int_{\mathbb{R}}-(2u_x+v_x)(2u-2u_{xx}+v-v_{xx})dx^{'}
\\
&=\int_{\mathbb{R}}-2(uv)_{x}+2(u_xv_x)_xdx^{'}=0.
\end{split}
\end{equation}

Let us consider the ordinary differential equation:
\begin{equation}\label{eq412}
  \left\{\begin{array}{l}
  q_t(t,x)=(2u+v)(t,q(t,x)),\quad t\in[0,T),  \\
  q(0,x)=x,\quad x\in\mathbb{R}.
  \end{array}\right.
\end{equation}
If $(m, n)\in B^s_{p,r} \times B^s_{p,r}$ with $(s,p,r)$ being as in Theorem \ref{theorem}, then $2u+v\in C([0,T); C^{0,1})$. According to the classical results in the theory of ordinary differential equations, we can easily infer that \eqref{eq412} have a unique solution $q\in C^1([0,T)\times\mathbb{R};\mathbb{R})$ such that the map $q(t,\cdot)$ is an increasing diffeomorphism of $\mathbb{R}$ with
$$
q_x(t,x)=\exp\Big(\int_0^t (2u_x+v_x)(t',q(t',x)\Big)dt'>0,\quad \forall (t,x)\in[0,T)\times\mathbb{R}.
$$

\begin{theo}
Let $m_0,n_0\in B^{s}_{p,r}\bigcap L^{1}$ with $(s,p,r)$ satisfies Theorem \ref{theoxx}, and $ m_0 \geq 0, n_0\geq 0$.
Then the corresponding solution $m,n$ of \eqref{eq1} exists globally in time.
\end{theo}

\textbf{Proof:}
From \eqref{eq1}, we know
\begin{equation}
 \left\{
  \begin{array}{l}
 \frac{d}{dt}m(t,q(t,x))=-3\big((2u_x+v_x)m\big)(t,q(t,x)),\\[1ex]
\frac{d}{dt}n(t,q(t,x))=-2\big((2u_x+v_x)m\big)(t,q(t,x)),\\[1ex]
  \end{array}
 \right.
\end{equation}
hence
\begin{equation}\label{eq414}
\left\{
\begin{array}{l}
m(t,q(t,x))=m_{0}(x)e^{\int_{0}^{t}\big(-3(2u_x+v_x)(t^{'},q(t^{'},x))\big)dt^{'}},\\
n(t,q(t,x))=n_{0}(x)e^{\int_{0}^{t}\big(-2(2u_x+v_x)(t^{'},q(t^{'},x))\big)dt^{'}},\\
\end{array}
\right.
\end{equation}
which implies that $(m,n)$ doesn't change sign. That is, when $m_0\geq 0, n_0\geq 0$, we have $ m\geq 0, n\geq 0$.
Moreover, we get
\begin{equation}
\begin{split}
|u_{x}|=|p_{x}*m(t,x)|\leq |p_{x}|*m(t,x)=p*m(t,x)=u(t,x)\Rightarrow |u_x|_{L^\infty}\leq |u|_{L^\infty},
\\
\|u_x+v_x\|_{L^\infty}\leq \|u+v\|_{L^\infty}=\|p*(m+n)\|_{L^\infty}\leq C\|m+n\|_{L^1}= C\|m_0+n_0\|_{L^1}\overset{\Delta}{=}C_0.
\end{split}
\end{equation}

Using integration by parts, we deduce that
\begin{equation}
\begin{split}
\frac{d}{dt}\int_{\mathbb{R}} m dx&=\int_{\mathbb{R}} -\Big((2u+v)m_x+3(2u_x+v_x)m\Big) dx \\
&= \int_{\mathbb{R}} -2(2u_x+v_x) m dx=\int_{\mathbb{R}} -4(u_x+v_x) m dx
\\
&\leq 4\|u_x+v_x\|_{L^\infty}\|m\|_{L^1} \leq4C_0 \|m\|_{L^1}.
\end{split}
\end{equation}
Applying Gronwall's inequality yields that, for all $t\in [0,T)$,
\begin{equation}
\|m\|_{L^1}\leq \|m_0\|_{L^1} e^{4C_0T},
\end{equation}
similarly,
\begin{equation}
\|n\|_{L^1}\leq \|n_0\|_{L^1} e^{2C_0T}.
\end{equation}
Combining (4.15), (4.17) and (4.18), we get
\begin{equation*}
\begin{split}
\|u_{x}\|_{L^{\infty}}\leq\|u\|_{L^{\infty}}\leq \|p_{x}*m\|_{L^{\infty}}\leq C\|m\|_{L^{1}}\leq \|m_0\|_{L^1} e^{4C_0T},
\\
\|v_{x}\|_{L^{\infty}}\leq\|v\|_{L^{\infty}}\leq \|p_{x}*n\|_{L^{\infty}}\leq C\|n\|_{L^{1}}\leq \|n_0\|_{L^1} e^{2C_0T}.
\end{split}
\end{equation*}
Moreover, we have
\begin{equation*}
\|u_x\|_{L^\infty}+\|v_x\|_{L^\infty} +\|u\|_{L^\infty}+\|v\|_{L^\infty}\leq 2\|n_0\|_{L^1} e^{2C_0T}+2\|m_0\|_{L^1} e^{4C_0T}.
\end{equation*}
Combining Theorem 4.1 then allows us to complete the proof of Theorem 4.3.

\begin{theo}
Let $(m, n)$ be the corresponding local solution of \eqref{eq1}, if the initial data $m_0, n_0\in B^s_{p,r} \bigcap L^1$ with $(s,p,r)$ satisfies Theorem \ref{theoxx}, $m_0(x)$ and $n_0(x)$ are odd functions such that $m_0\leq 0$, $n_0\leq 0$ when $x\leq 0$, and $m_0\geq 0$, $n_0\geq 0$ when $x\geq 0$, then the solution $m(t,x), n(t,x)$ exists globally.
\end{theo}

\begin{proof}
Set $T$ be the maximal time of $m(t,x)$ and $n(t,x)$.
One can easily deduce that $m(t,x)$ and $n(t,x)$ are odd functions if $m_0(x)$ and $n_0(x)$ are odd functions.
According to \eqref{eq414}, we can deduce that if $m_0\leq 0$ and $n_0\leq 0$ when $x\leq 0$, $m_0\geq 0$ and $n_0\geq 0$ when $x\geq 0$, then we obtain
\begin{align}\label{g1}
m(t,x)\leq 0, n(t,x)\leq 0~~when~~ x\leq q(t,0); ~~ m(t,x) \geq 0, n(t,x)\geq 0~~ when~~ x\geq q(t,0),~~\forall t\in [0,T),
\end{align}
where $q(t,\xi)=\xi+\int_{0}^{t}(2u+v)(\tau,q(\tau,\xi))d\tau$ is the characteristic curves.

Next we want to prove that
\begin{align}\label{g2}
m(t,x)\leq 0, n(t,x)\leq 0~~when~~ x\leq 0; ~~ m(t,x) \geq 0, n(t,x)\geq 0~~ when~~ x\geq 0,~~\forall t\in [0,T).
\end{align}
If $q(t,0)=0$ for any $t\in[0, T)$, then by \eqref{g1} we immediately get \eqref{g2}. Other if $q(t,0)>0$ (or $q(t,0)<0$) for some $t\in[0, T)$, by \eqref{g1} we deduce that
$$
m(t,x)\leq 0, n(t,x)\leq 0~~when~~ x\in [-q(t,0),0]; ~~ m(t,x) \geq 0, n(t,x)\geq 0~~ when~~ x\in [0,q(t,0)].
$$
Since $m(t,x), n(t,x)$ are odd functions such that $m(t,x)=-m(t,-x), n(t,x)=-n(t,-x),~x\in [-q(t,0),q(t,0)]$, we obtain
$$
m(t,x)= 0, n(t,x)= 0~~when~~ x\in [-q(t,0),q(t,0)].
$$
By \eqref{g1} again we still get \eqref{g2}. Note that the function $u(t,0)=m(t,0)=0$, $v(t,0)=n(t,0)=0$.
Using the characteristic method and calculations similar to Theorem 4.2, we have
$$\frac{d}{dt} \int_{0}^{+\infty}(m+n)dx^{'}=0,$$
then
$$\frac{d}{dt} \int_{\mathbb{R}}|m|+|n|dx^{'}=\frac{d}{dt} \int_{\mathbb{R}}|m+n|dx^{'}=2\frac{d}{dt} \int_{0}^{+\infty}(m+n)dx^{'}=0,$$
thus,
$$\|m\|_{L^1}+\|n\|_{L^1}=\|m_0\|_{L^1}+\|n_0\|_{L^1}.$$
Similar to the proof of Theorem 4.3, we get
\begin{equation}
\begin{split}
&\|u\|_{L^\infty}+\|u_x\|_{L^\infty}\leq \|p*m\|_{L^\infty}+\|p_x*m\|_{L^\infty}\leq C\|m\|_{L^1}\leq C (\|m_0\|_{L^1}+\|n_0\|_{L^1}),
\\
&\|v\|_{L^\infty}+\|v_x\|_{L^\infty}\leq \|p*n\|_{L^\infty}+\|p_x*n\|_{L^\infty}\leq C\|n\|_{L^1}\leq C (\|m_0\|_{L^1}+\|n_0\|_{L^1}).
\end{split}
\end{equation}
According to (4.21), we have
\begin{equation*}
\|u_x\|_{L^\infty}+\|v_x\|_{L^\infty} +\|u\|_{L^\infty}+\|v\|_{L^\infty}\leq C(\|m_0\|_{L^1}+\|n_0\|_{L^1}).
\end{equation*}
By the blow-up criteria we obtain the global existence.
\end{proof}

\section{Conclusion}

In this paper, we have studied the local well-posedness of Popowicz system in  Besov spaces $B^s_{p,r}\times B^s_{p,r}$ with $s> \max\{2, \frac{1}{p}+\frac{3}{2}\}$
or $(s=2, 2\leq p \leq \infty, 1\leq r\leq 2)$,
which is an extension of Besov spaces $B^s_{p,r}$ with $1\leq p, r\leq +\infty$ and $s>\max\{\frac{5}{2}, 2+\frac{1}{p}\}$ in \cite{3}.
We also obtained a new blow up criterion, which is an extension of \cite{4} and \cite{5}.
Besides, we studied the global existence with different initial values of Popowicz systems for the first time.
As the interacting system of C-H and D-P equations, Popowicz system has many properties worthy of further study, such as blow-up of solutions \cite{15}, existence of weak solutions \cite{16}, global conservation weak solutions \cite{19}, dynamic behavior of solitons \cite{Tan} and ill-posedness \cite{Guo},  which will be studied in our future work.

\smallskip
\noindent\textbf{Acknowledgments} This work was
This work was partially supported by NNSFC (Nos. 12171493, 11671407),FDCT (No. 0091/20181A3), the Natural Science Foundation of Hunan Province (No. 2021JJ40434), and
the Scientific Research Project of the Hunan Education Department (No. 21B0510).


\phantomsection
\addcontentsline{toc}{section}{\refname}
\bibliographystyle{abbrv} 
\bibliography{Feneref}

\end{document}